\documentclass{article}
\usepackage{amsfonts}
\usepackage{amsthm}

\usepackage{amssymb,stmaryrd,mathrsfs}
\usepackage{amsmath,latexsym,mathdots}

\usepackage{hyperref}
\usepackage{manfnt}
\usepackage{verbatim}
\usepackage{frcursive}
\usepackage[applemac]{inputenc}

\newcommand{\R}{{\mathbb R}}

\newcommand{\N}{{\mathbb N}}
\newcommand{\C}{{\mathbb C}}
\newcommand{\Z}{{\mathbb Z}}
\newcommand{\HH}{{\mathbb H}}
\newcommand{\dif}{{\rm d}}
\newcommand{\Dif}{{\rm D}}
\newcommand{\Real}{{\sf Re}}
\newcommand{\Imag}{{\sf Im}}

\newcommand{\Euler}{{\sf E}}
\newcommand{\transp}{{\sf T}}
\newcommand{\Legendre}{{\sf L}}
\newcommand{\Hess}{{\sf Hess}}
\newcommand{\Lag}{{\mathscr{L}}}
\newcommand{\lag}{{\ell}}
\newcommand{\emm}{{m}}
\newcommand{\vits}{{u}}
\newcommand{\vitsL}{{w}}
\newcommand{\vol}{{v}}

\newcommand{\En}{\mathscr{E}}
\newcommand{\Ec}{\mathscr{T}}
\newcommand{\en}{\mbox{{\begin{cursive}{\emph{e}}\end{cursive}}}}
\newcommand{\funct}{{\mathscr{F}}}
\newcommand{\ham}{{{h}}}
\newcommand{\Ham}{{\mathscr{H}}}

\newcommand{\speed}{{c}}
\newcommand{\wn}{{{k}}}
\newcommand{\freq}{{\omega}}
\newcommand{\meanU}{{{\bf M}}}
\newcommand{\meanI}{{P}}

\newcommand{\momB}{{\sf M}}
\newcommand{\Impulse}{{\mathscr{Q}}}
\newcommand{\Impulseflux}{{\mathscr{S}}}
\newcommand{\Evans}{{D}}

\newcommand{\bu}{{u}}
\newcommand{\ubu}{\underline{u}}
\newcommand{\uq}{\underline{\bf q}}
\newcommand{\uv}{\underline{v}}
\newcommand{\bU}{{\bf U}}
\newcommand{\bV}{{\bf V}}
\newcommand{\bW}{{\bf W}}
\newcommand{\bC}{{\bf C}}

\newcommand{\ubM}{\underline{\bf M}}
\newcommand{\ubU}{\underline{\bf U}}

\newcommand{\uU}{\underline{U}}
\newcommand{\sJ}{{{\mathscr J}}}

\newcommand{\bS}{{\bf S}}
\newcommand{\bB}{{\bf B}}
\newcommand{\bT}{{T}}
\newcommand{\bb}{{b}}

\newcommand{\bq}{{\bf q}}
\newcommand{\bff}{{f}}
\newcommand{\bm}{{\bf m}}
\newcommand{\bSigma}{{\boldsymbol{\Sigma}}}
\newcommand{\bPi}{{\boldsymbol{\Pi}}}
\newcommand{\blambda}{{\boldsymbol{\lambda}}}

\newcommand{\vnu}{{\vnu}}

\renewcommand{\neg}{{{\sf n}}}
\newcommand{\Linvar}{{{\mathscr{A}}}}
\newcommand{\linvar}{{{\bf a}}}
\newcommand{\Fund}{{{\bf F}}}

\newcommand{\Lin}{{{\bf A}}}
\newcommand{\Linco}{{{\mathscr{A}}^{0}}}

\newcommand{\ee}{{\rm e}}

\newtheorem{corollary}{Corollary}
\newtheorem{proposition}{Proposition}
\newtheorem{theorem}{Theorem}
\newtheorem{remark}{Remark}

\title 
{Stability of periodic waves in Hamiltonian PDEs}

\author{Sylvie Benzoni-Gavage\thanks{benzoni@math.univ-lyon1.fr}, Pascal Noble\thanks{Pascal.Noble@math.univ-toulouse.fr}, and L.Miguel Rodrigues\thanks{rodrigues@math.univ-lyon1.fr}}

\begin{document}

\maketitle

\begin{abstract}
Partial differential equations endowed with a Hamiltonian structure, like the Korteweg--de Vries equation and many other more or less classical models, are known to admit rich families of periodic travelling waves. The stability theory for these waves is still in its infancy though. The issue has been tackled by various means. Of course, it is always possible to address stability from the spectral point of view. However, the link with nonlinear stability ~-~in fact, \emph{orbital} stability, since we are dealing with space-invariant problems~-~, is far from being straightforward when the best spectral stability we can expect is a \emph{neutral} one. Indeed, because of the Hamiltonian structure, the spectrum of the linearized equations cannot be bounded away from the imaginary axis, even if we manage to deal with the point zero, which is always present because of space invariance.
Some other means make a crucial use of the underlying structure. This is clearly the case for the variational approach, which basically uses the Hamiltonian -~or more precisely, a  constrained functional associated with the Hamiltonian and with other conserved quantities~- as a Lyapunov function. When it works, it is very powerful, since it gives a straight path to orbital stability. An alternative is the modulational approach, following the ideas developed by Whitham almost fifty years ago. The main purpose here is to point out a few results, for KdV-like equations and systems, that make the connection between these three approaches: spectral, variational, and modulational.
\end{abstract}

%\thanks
\paragraph{Acknowledgments} This work has been partly supported by
the European Research Council \href{http://math.univ-lyon1.fr/~filbet/nusikimo/nusikimo.htm}{ERC Starting Grant 2009, project 239983- NuSiKiMo}, and the Agence Nationale de la Recherche, JCJC project {Shallow Water Equations for Complex Fluids} 2009-2103. This paper has been prepared for the Proceedings of \emph{Journ\'ees \'Equations aux Dérivées Partielles}, \href{http://gdredp.math.cnrs.fr/}{GDR CNRS 2434}, \href{http://gdredp.math.cnrs.fr/spip/spip.php?article167}{Biarritz 2013}.
%\email{benzoni@math.univ-lyon1.fr}

%\keywords
  
%\altkeywords{onde progressive périodique, stabilité variationnelle, stabilité spectrale, stabilité modulationnelle}

%\subjclass
\paragraph{Classification} {35B10; 35B35;  35Q35; 35Q51; 35Q53; 37K05; 37K45.}

\section{Introduction}
Sound and light are manifestations of periodic waves, even though they are hardly perceived as waves in daily life. Perhaps the most famous, clearly visible periodic waves are those propagating at the surface of water, named after  George Gabriel {Stokes}. In real-world situations, periodic water waves can be formed for instance by ships. Their two main features are \emph{non-linearity} and \emph{dispersion}, which imply that their velocity depends on both their amplitude and their wavelength.
However, it was observed in a celebrated work \cite{BenjaminFeir} that the so-called Stokes waves were not so easy to create in lab experiments. At first puzzled by this problem, Benjamin and Feir exhibited a threshold for the ratio of depth over wave length above which small amplitude Stokes waves become unstable. 

If the Stokes waves are an archetype of nonlinear dispersive waves, the underlying -~water wave~- equations are quite complicated. The purpose of this talk was to give an overview of stability theory for a wide range of nonlinear dispersive waves, of possibly arbitrary amplitude, arising as solutions of PDEs endowed with a `nice' algebraic structure. This has been a renewed, active field in the last decade, with still a number of open questions even in one space dimension.  By contrast, the theory is much more advanced regarding solitary waves, which may be viewed as a limiting case of periodic waves -~namely, when their wavelength goes to infinity.

We restrict to one-dimensional issues in what follows. In mathematical physics, there are a number of model equations supporting nonlinear dispersive waves. The most classical ones are known as 
\emph{the} {N}on-{L}inear {W}ave equation
$$\mbox{(NLW)}\quad \partial_t^2\chi - \partial_x^2\chi + v(\chi)=0\,, $$
the (generalized) Boussinesq equation
%$$\partial_t^2\chi - \partial_x(w(\partial_x\chi))=0$$
$$\mbox{(B)}\quad \partial_t^2\phi - \partial_x^2(w(\phi) \mp \partial_x^2 \phi)=0\,, $$
the (generalized) {Korteweg}-{de} {Vries} equation
$$\mbox{(KdV)}\quad \partial_tv+\partial_xp(v)=-\partial_x^3v\,,$$
and the {N}on-{L}inear {Schrödinger} equation
$$\mbox{(NLS)}\quad i \partial_t\psi +\tfrac{1}{2} \partial_x^2 \psi\,=\,\psi \,g(|\psi|^2)\,.$$
It is on purpose that we have chosen to write non-linear terms in their most general form here above -~observe that nonlinearities are written as $v(\chi)$ in (NLW), $w(\phi)$ in (B), $p(v)$ in (KdV), and $\psi \,g(|\psi|^2)$ in (NLS). 
As a matter of fact, we shall refrain from invoking integrability arguments, which only work for some specific nonlinearities.
Nevertheless, a common feature of these equations is that they are endowed with a Hamiltonian structure. Indeed, they can all be written in the abstract form
\begin{equation}\label{eq:absHam}\partial_t\bU = \sJ (\Euler \Ham[\bU])\,,\end{equation}
where the unknown $\bU$ takes values in $\R^N$ ($N=1$ for (KdV), $N=2$ for (B), (NLS), $N=3$ for (NLW)), $\sJ$ is a skew-adjoint differential operator, and $\Euler \Ham$ denotes the variational derivative of $\Ham$, whose $\alpha$-th component ($\alpha\in \{1,\ldots,N\}$) merely reads as follows when $\Ham=\Ham(\bU,\bU_x)$,
 $$(\Euler \Ham[\bU])_\alpha\,:=\,\frac{\partial \Ham }{\partial U_\alpha}(\bU,\bU_x)\,-\,\Dif_{x}\left(\frac{\partial \Ham }{\partial U_{\alpha,x}}(\bU,\bU_x)\right).$$
 Here above, $\Dif_{x}$ stands for the \emph{total} derivative. More explicitly, this means that
$$\Dif_{x}\left(\frac{\partial \Ham }{\partial U_{\alpha,x}}(\bU,\bU_x)\right)\,=\,\frac{\partial^2 \Ham(\bU,\bU_x) }{\partial U_\beta \partial U_{\alpha,x}}\,U_{\beta,x}\,+\,\frac{\partial^2 \Ham(\bU,\bU_x) }{\partial U_{\beta,x} \partial U_{\alpha,x}}\,U_{\beta,xx}\,,
$$
where we have used Einstein's convention of summation over repeated indices.  Another convention is that square brackets $[\cdot]$ signal a function of not only the dependent variable $\bU$ but also of its derivatives $\bU_x$, $\bU_{xx}$, \ldots (For instance, we shall either write $\Ham(\bU,\bU_x)$ or $\Ham[\bU]$.)
A motivation for addressing the stability of periodic waves in such an abstract setting is to make the most of algebra, irrespective of the %specific 
model under consideration. However, we do have a specific model in mind, namely 
the Euler--Korteweg system, which admits two different formulations depending on whether we choose Eulerian coordinates, 
\begin{equation*}\label{eq:EKabs1d}
\mbox{(EKE)}\quad \left\{\begin{array}{l}\partial_t\rho +\partial_x (\rho \vits)\,=\,0\,,\\ [5pt]
\partial_t \vits + \vits\partial_x\vits \,+\,\partial_x(\Euler_\rho \En )\,=\,0\,,\quad \En=\En(\rho,\rho_x)\,,
\end{array}\right.
\end{equation*}
or Lagrangian coordinates, 
\begin{equation*}\label{eq:EKabsLagb}
\mbox{(EKL)}\quad \left\{\begin{array}{l}\partial_t{\vol} \,=\,\partial_y{\vits}\,,\\ [5pt]
\partial_t {\vits} \,=\, \partial_y( \Euler_\vol \en)\,,\quad \en=\en(\vol,\vol_y)\,,
\end{array}\right.
\end{equation*}
both fitting the abstract framework in \eqref{eq:absHam}. (For details on all these equations, see Table in Appendix.) This is not that a specific model though. What we call the Euler--Korteweg system comprises many models of mathematical physics, including the Boussinesq equation for water waves, as well as (NLS) after Madelung's transformation, see for instance \cite{SBG-DIE} for more details. 

In the literature on Hamiltonian PDEs, the distinction is often made between `NLS-like equations', in which $\sJ$ is merely a real skew-symmetric matrix, and `KdV-like equations', in which $\sJ=\bB\partial_x$ with $\bB$ a real symmetric matrix. This distinction is to some extent artificial, since for instance (NLS) can be written as a special case of the KdV-like system (EKE), and on the contrary (EKE) can take the form of a NLS-like system if the hydrodynamic potential is to replace the velocity $\vits$ as a dependent variable. However,  there should be  a most `natural' formulation for each equation or system.

From now on, we concentrate on KdV-like equations, 
%the authors owe a lot to the literature on gKdV...
and assume that $\sJ=\bB\partial_x$ with $\bB$ a nonsingular, symmetric matrix. In this case, \eqref{eq:absHam} is itself a system of conservation laws, which reads
\begin{equation}\label{eq:absHamb}\partial_t\bU = \partial_x (\bB\,\Euler \Ham[\bU])\,,\end{equation}
and turns out to admit the additional, scalar conservation law
\begin{equation}
\label{eq:impulsecl}
\partial_t\Impulse(\bU) = \partial_x( \Impulseflux[\bU])\,
\end{equation}
with $$\Impulse(\bU):= \tfrac{1}{2}\, \bU\cdot \bB^{-1} \bU\,,\; \Impulseflux[\bU]\,:=\,\bU \cdot \Euler \Ham[\bU]\,+\,\Legendre \Ham[\bU]\,,$$
$$\Legendre \Ham[\bU]\,:=\,
 U_{\alpha,x}\,\frac{\partial \Ham}{\partial  U_{\alpha,x}}(\bU,\bU_{x}) - \Ham(\bU,\bU_{x})\,.$$
The dots $\cdot$ in the definitions of $\Impulse$ and $\Impulseflux$ are for the `canonical' inner product 
 $\bU\cdot\bV=U_\alpha V_\alpha$ in $\R^N$.
The letter $\Legendre$ stands for the `Legendre transform' (even though it is considered in the original variables $(\bU,\bU_{x})$).
Equation \eqref{eq:impulsecl} is satisfied along any smooth solution of \eqref{eq:absHam}.
Notice that for any (smooth) function $\bU$,
\begin{equation}\label{eq:translx}
%\sJ (\Euler \Impulse[\bU]) = \partial_x \bU\,,
\partial_x \bU = \partial_x (\bB\,\Euler \Impulse[\bU]) \,.
\end{equation}
Viewed as $\partial_x \bU=\sJ (\Euler \Impulse[\bU])$, this relation
reveals that the (local) conservation law \eqref{eq:impulsecl} for $\Impulse(\bU)$ is associated with the invariance of \eqref{eq:absHamb} under spatial translations. Any such quantity\footnote{which also exist for NLS-like equations, but are no longer algebraic and  depend on $\bU_x$, see Table in Appendix.} has been called an \emph{impulse} by Benjamin \cite{Benjamin}. Of course there is also a conservation law associated with the invariance of \eqref{eq:absHam} under time translations, which is nothing but the (local) conservation law for the Hamiltonian
\begin{equation}
\label{eq:Hamcl}
\partial_t\Ham(\bU,\bU_x) = \partial_x\Big(\tfrac{1}{2} \Euler\Ham[\bU]\cdot \bB \Euler\Ham[\bU]+ \nabla_{\bU_x}\Ham[\bU]\cdot \Dif_x( \Euler\Ham[\bU]) \Big)\,.
\end{equation}
However, this rather complicated conservation law will play a much less prominent role than \eqref{eq:impulsecl} in what follows.

For a travelling wave $\bU=\ubU(x-\speed t)$ of speed $\speed$ to be solution to \eqref{eq:absHam}, one must have 
by \eqref{eq:translx}  that
$$\partial_x(\Euler (\Ham+\speed\Impulse)[\ubU])=0\,,$$
or equivalently, there must exist $\blambda\in \R^N$ such that
\begin{equation}
\label{eq:EL}
\Euler (\Ham+\speed\Impulse)[\ubU] \,+\,\blambda\,=\,0\,.
\end{equation}
%For consistency of notations, let us introduce 
This is nothing but the Euler-Lagrange equation associated with the \emph{Lagrangian}
$$\Lag= \Lag(\bU,\bU_x;\speed,\blambda):= \Ham(\bU,\bU_x)+\speed\Impulse(\bU)\,+\,\blambda\cdot \bU\,.$$
As is well-known, an Euler-Lagrange equation for a Lagrangian $\Lag$ admits $\Legendre\Lag$~-~the `Legendre transform' of $\Lag$~-~ as a first integral. Unsurprisingly, this first integral coincides here with 
$\Impulseflux+\speed \Impulse$, a quantity that is clearly constant along the travelling wave, thanks to \eqref{eq:impulsecl}. The reader may easily check indeed that 
$$\Legendre\Lag[\ubU]\,=\, \Impulseflux[\ubU]+\speed \Impulse[\ubU]$$
as soon as \eqref{eq:EL} holds true.
Therefore, a full set of equations for the travelling profile $\ubU$ consists of  \eqref{eq:EL} together with
\begin{equation}
\label{eq:ELham}
\Legendre\Lag[\ubU]\,=\,\mu\,,
\end{equation}
where $\mu$ is a constant of integration.
Recalling that $\Lag$ depends on $(\speed, \blambda)$, we see that a travelling profile $\ubU$ depends %generically 
on $(\speed, \blambda, \mu)\in \R^{N+2}$, which `generically' makes the set of profiles an $(N+2)$-dimensional manifold. This is up to translations of course, because any translated version $x\mapsto \ubU(x+s)$ (for an arbitrary $s\in \R$) of $\ubU$ still solves  \eqref{eq:EL}-\eqref{eq:ELham}.

In practice, the \emph{existence} of periodic waves is not straightforward. However, it almost becomes so if $N=1$ or $2$, under a few assumptions that are met by all our KdV-like equations (namely, (KdV) itself, (EKE), and (EKL)).
The simplest case is $N=1$, with the dependent variable $\bU$ being reduced to a scalar variable $v$, and
$$\Ham=\En(v,v_x)\,,\;\frac{\partial^2 \Ham}{\partial v_x^2}=\frac{\partial^2 \En}{\partial v_x^2}=:\kappa(v)>0\,.$$
(This is a slight generalization of what happens with the usual KdV-equation, in which $\kappa$ is constant.)
A little more complicated case is with $N=2$, with the dependent variable $\bU=(v,u)$, and 
\begin{equation}\label{structure1}
\Ham= \Ham(v,\bu,v_x)\,=\,\En(v,v_x)\,+\,\Ec(v,\bu)\,,
\end{equation}
such that 
\begin{equation}\label{structure2}
\frac{\partial^2 \Ham}{\partial v_x^2}=\frac{\partial^2 \En}{\partial v_x^2}=:\kappa(v)>0\,,\quad 
\frac{\partial^2\Ham}{\partial {\bu}^2}=
\frac{\partial^2\Ec}{\partial {\bu}^2}=:\bT(v)>0\,.
\end{equation}
\begin{equation}\label{structure3}
\bB^{-1}=\left(\begin{array}{cc} a & \bb 
\\ \bb & 0
\end{array}\right)\,,\quad b\neq 0\,.
\end{equation}
(These assumptions are met by both (EKE) and (EKL).)
In this way, we may eliminate $\ubu$ from the profile equations \eqref{eq:EL} and receive a single, second order ODE in $\uv$, which also inherits a Hamiltonian structure, and is therefore \emph{completely integrable}.
The reader might want to see this equation. Otherwise, they may skip what follows and go straight to the end of this section. 

The second component in \eqref{eq:EL} reads indeed
$$\bT(\uv)\,\ubu+
\partial_{\bu} \Ec(\uv,0)+\speed\, \uv\, \bb +
\lambda_2=0\,,$$
where
$\lambda_2$ is the second component of $\blambda$.
Since $\bT(\uv)$ is nonzero, this gives 
$$\ubu=\bff(\uv;\speed,
\lambda_2):=-\bT(\uv)^{-1}\,(\partial_{\bu} \Ec(\uv,0)+\speed\, \uv\, \bb +
\lambda_2)\,.$$
By plugging this expression in \eqref{eq:ELham}, we arrive at
$$
\Legendre \En[\uv]\,-
\Ec(\uv,\bff(\uv;\speed,
\lambda_2))\,-\,
\speed \Big(\tfrac{1}{2} a \uv^2\,+\,\uv \bb
\bff(\uv;\speed,
\lambda_2)\Big)\,-\,\lambda_1\,\uv \,-\,
\lambda_2\, \bff(\uv;\speed,
\lambda_2)\,=\,\mu\,.$$
Despite its terrible aspect,  this equation is merely of the form
\begin{equation}\label{eq:profv}
\frac{1}{2}\kappa(\uv)\uv_x^2\,+\,W(\uv;\speed,\blambda)\,=\,\mu\,,
\end{equation}
if $\En$ is really quadratic in $v_x$ (\emph{i.e.} if $\partial_{v_x}\En(v,0)=0$). 
We obtain a similar one in the case $N=1$ with $\Ham=\En(v,v_x)$.
Eq.~\eqref{eq:profv} can be viewed as an integrated version of the Euler--Lagrange ODE, $\Euler\lag=0$, associated with the `reduced' Lagrangian
$$\lag:=\frac{1}{2}\kappa(\uv)\uv_x^2\,-\,W(\uv;\speed,\blambda)\,.$$
Incidentally, $\Euler\lag=0$ admits as a first integral the `reduced' Hamiltonian
$$\ham:=\frac{1}{2}\kappa(\uv)\uv_x^2\,+\,W(\uv;\speed,\blambda)\,.$$
We thus find families of periodic orbits  parametrized by $\mu$ around any local minimum of the potential $W(\cdot;\speed,\blambda)$. In case $W(\cdot;\speed,\blambda)$ is a double-well potential, which is what happens with the famous van der Waals/Cahn--Hilliard/Wilson energies, a same parameter $\mu$ can clearly be associated with two different orbits. In other words, the whole set of periodic orbits is not made of a single graph over the set of parameters $(\mu,\blambda,\speed)$. Nevertheless, each family of periodic orbits
can be parametrized by $(\mu,\speed,\blambda)$, as long as the wells of 
$W(\cdot;\speed,\blambda)$  remain distinct.

%\paragraph{}\label{par:solitons}
Going back to the more comfortable general setting, let us just assume that there exist open sets of parameters $(\mu, \blambda, \speed)$ for which \eqref{eq:EL}-\eqref{eq:ELham} have a unique periodic solution up to translations. Note that the set of \emph{solitary} wave profiles may be viewed as a co-dimension one boundary of periodic profiles. Indeed, for a solitary wave profile, once $\blambda$ has been prescribed by the endstate $\bU_\infty=(v_\infty,\bu_\infty)$,
$$\blambda\,=\,-\,\nabla_{\bU} (\Ham+\speed\Impulse)(\bU_\infty,0)\,,$$
the constant of integration $\mu$ is given by 
$$\mu=- \Ham(v_\infty,\bu_\infty,0)\,-\,\speed\Impulse (\bU_\infty) \,-\,\blambda \cdot \bU_\infty\,.$$

We now aim at investigating the stability of periodic travelling waves $\bU=\ubU(x-\speed t)$.
For this purpose, 
some global, stringent assumptions --- for instance quadraticity in $v_x$ and $\bu$ --- may 
often be  
relaxed to suitable, local invertibility assumptions.

\section{Various types of stability}
Let us consider a periodic travelling wave $\bU=\ubU(x-\speed t)$ solution to \eqref{eq:absHam}. In other words, we assume that $\ubU$ is a periodic solution to \eqref{eq:EL}-\eqref{eq:ELham}, and denote by %$\uXi$
$\Xi$ its period\footnote{Please note that this is a \emph{spatial} period. We refrain from using the word `wavelength' here in order to prevent the reader from thinking $\bU$ as a harmonic wave. It can be a cnoidal wave, or any kind of periodic wave.}. The latter is supposed to be uniquely determined, say in the vicinity of a reference profile, by the parameters $(\mu, \blambda, \speed)$.
As to the profile $\ubU$, it can only be unique up to translations. 
Thus, we may assume without loss of generality that $\uv_x(0)=0$. This choice will play a role in subsequent calculations.
Let us now review a series of related notions and tools.

\subsection{Variational point of view}\label{ss:var}
By the Euler--Lagrange equation in \eqref{eq:EL}, $\ubU$ is a critical point of the functional
$$\funct^{(\speed,\blambda,\mu)}: \bU\,\mapsto \,\int_{0}^{\Xi} (\Ham(\bU,\bU_x)+\speed \Impulse(\bU) + \blambda\cdot \bU +\mu)\,\dif x \,.$$
(At this point, the $\mu$ term does not play any role but it will come into play later on.)
Would in addition $$\Theta(\mu,\blambda,\speed):= \funct^{(\speed,\blambda,\mu)}[\ubU]=\int_{0}^{\Xi} (\Ham(\ubU,\ubU_x)+\speed \Impulse(\ubU) + \lambda\cdot \ubU +\mu)\,\dif x$$ be a (locally) minimal value of $\funct^{(\speed,\blambda,\mu)}$,
it would be natural to use this functional as a \emph{Lyapunov} function in order to show the stability of $\ubU$. This would require, though, that its Hessian,
$$\Linvar:=\Hess(\Ham+\speed \Impulse)[\ubU]$$
be a positive differential operator. (Of course $\Linvar$ depends on the parameters $(\mu, \blambda, \speed)$ but we omit to write them in order to keep the notation simple.) This we would call \emph{variational} stability. However, 
there is no hope that it be the case. A first reason is, by differentiating \eqref{eq:EL} with respect to $\speed$, 
we readily see that $\Linvar \ubU_x = 0$. 
 Hence $\Linvar$ has a nontrivial kernel on $L^2(\R/\Xi\Z)$, containing at least $\ubU_x$, as is always the case with space-invariant problems.
An even worse observation is that, by a Sturm--Liouville argument applied to the second order ODE  satisfied by $\uv$, the equality $\Linvar \ubU_x = 0$ certainly implies that $\Linvar$ has a negative eigenvalue (see Appendix for more details). Nevertheless, what we can hope for is \emph{constrained} variational stability. Indeed, knowing that $\bU$ and $\Impulse(\bU)$ are conserved quantities, it can be that the
values of $\funct^{(\speed,\blambda,\mu)}$ which are lower than $\Theta(\mu,\blambda,\speed)$ are
not seen on the manifold 
$${\mathscr C}:=\{\bU\,;\;  \textstyle \int_{0}^{\Xi}\Impulse(\bU)\,\dif x =\int_{0}^{\Xi}\Impulse(\ubU)\,\dif x\,,
\int_{0}^{\Xi} \bU\,\dif x =\int_{0}^{\Xi} \ubU \,\dif x\}\,.$$
By `not seen' we mean an infinite-dimensional analogue of what happens for instance with the indefinite
function $(x,y)\mapsto y^2-x^2$, which does have a (local) minimum along any curve lying in $\{(0,0)\}\cup \{(x,y)\,;\;|x|<|y|\}$. 
Determining whether ${\mathscr C}$ is located in the `good' region amounts to identifying suitable inequalities, 
which may be viewed as generalizations of the Grillakis-Shatah-Strauss criterion known for solitary waves \cite{GSS}, as we shall explain in Section \ref{ss:GSS}. These inequalities should ensure that $\Linvar$ is nonnegative on the tangent space $T_{\ubU}{\mathscr C}$, a necessary condition for the functional $\funct^{(\speed,\blambda,\mu)}$ to be minimized at $\ubU$ along ${\mathscr C}$. Then we may speak of constrained variational stability \emph{despite} the translation-invariance problem, that is, even though $\ubU$ is not a strict minimizer.
Indeed, as observed in earlier work on solitary waves \cite[Lemma 3.2]{GSS}, 
any $\bU$ close to $\ubU$ admits  by the implicit function theorem a translate $
x\mapsto \bU(x+s(\bU))$ such that $
\bU(\cdot+s(\bU))-\ubU$ is orthogonal to $\ubU_x$ with respect to the $L^2$ inner product. 
This argument clearly paves the way towards \emph{orbital} stability. As a matter of fact, by reasoning as in \cite[Theorem 3.5]{GSS} with an appropriate choice of a function space $\HH\subset L^2(\R/\Xi \Z)$ in which we would have a flow map $\bU(0)\mapsto \bU(t)$ for \eqref{eq:absHam}, we might prove that
$$
\forall \varepsilon>0,\;\exists \delta>0\,;\;
\|\bU(0)\,-\,\ubU\|_\HH\,\leq \delta\; \Rightarrow \forall t\geq 0\,,\\
\displaystyle \inf_{s\in \R}\,\|\bU(t)\,-\,\ubU(\cdot+s)\|_\HH\,\leq\,\varepsilon\,.
$$
This would mean orbital stability of $\ubU$ with respect to \emph{co-periodic} perturbations ($\HH$ being made of $\Xi$-periodic functions). Possibly redefining $\funct^{(\speed,\blambda,\mu)}$ as an integral over an interval of length  $n\Xi$ for an integer $n\geq 2$, we might also prove orbital stability with respect to multiply periodic perturbations, that is in $L^2(\R/n\Xi\Z)$. Note however that this would require a more delicate count of signatures \cite{DeconinckKapitula}, because the negative spectrum of $\Linvar$ grows bigger when $n$ increases (again by a Sturm--Liouville argument).
As to `localized' perturbations, there is no obvious definition of a functional that would play the role of $\funct^{(\speed,\blambda,\mu)}$. This differs from the case of solitary waves, for which 
$${\mathscr M}^{(\bU_\infty,c)}:\bU\,\mapsto \,\int_{-\infty}^{\infty} (\Ham(\bU,\bU_x)+\speed \Impulse(\bU) + \lambda\cdot \bU +\mu)\,\dif x$$
does the job. If $\bU=\ubU(x-\speed t)$ is a solitary wave homoclinic to $\bU_\infty$, the integral $\momB(\bU_\infty,c):={\mathscr M}^{(\bU_\infty,c)}(\ubU)$ has been known as the \emph{Boussinesq moment of instability}, and the Grillakis-Shatah-Strauss criterion requires that 
$$\frac{\partial^2 \momB}{\partial c^2} =:\momB_{cc}>0$$
for this wave to be stable \cite{Boussinesq,Benjamin72,BonaSachs,SBG-DIE}. The reason why it is the sign of $\momB_{cc}$ that plays a role in the solitary wave stability is not difficult to see, as soon as we have in mind the following crucial relations,
$$\Linvar \ubU_c\,=\,\nabla \Impulse(\bU_\infty)-\nabla\Impulse(\ubU)=:\uq\,,\quad \momB_{cc}=-\langle \uq\cdot \ubU_c\rangle_{L^2}\,,$$
obtained by differentiating the profile equation \eqref{eq:EL} with respect to $c$ at fixed $\bU_\infty$, and of course also $\momB$. Assuming that $\momB_{cc}$ is nonzero, we thus see that any $\bU\in {\mathscr D}(\Linvar)$ can be decomposed in a unique way as $\bU=a \ubU_c+\bV$ with 
$\langle \uq\cdot\bV\rangle_{L^2}=0$, 
and
$$\langle \Linvar \bU\cdot \bU\rangle_{L^2}\,=\,-\,a^2\,\momB_{cc}\,+\,\langle \Linvar \bV\cdot \bV\rangle_{L^2}\,.$$
On this identity we see that the negative signature $\neg(\Linvar)$ of $\Linvar$ equals the one 
$\neg(\Linvar_{|\uq^\perp})$ 
of~$\Linvar_{|\uq^\perp}$ if $\momB_{cc}<0$, whereas 
$$\neg(\Linvar)=\neg(\Linvar_{|\uq^\perp})+1$$
if $\momB_{cc}>0$. In the latter situation, if it is true that $\Linvar$ has a single negative eigenvalue, we find that $\Linvar_{|\uq^\perp}$ has no negative spectrum, hence constrained variational stability. (The proof of orbital stability then follows by a contradiction argument \cite{GSS,BSS}.) On the other hand, $\Linvar_{|\uq^\perp}$ does have negative spectrum if  $\momB_{cc}<0$, hence constrained variational instability. (The proof in \cite{GSS} that this implies orbital instability is trickier, and does not work if we cannot assure that there is a negative direction $y$ of $\Linvar_{|\uq^\perp}$ in the range of $\sJ$, which is equivalent to requiring that $\int_{-\infty}^{+\infty} y \dif x=0$
if $\sJ= \bB\partial_x$. This issue was fixed in \cite{BSS} for (KdV).)

Let us go back to periodic waves. The functional $\funct^{(\speed,\blambda,\mu)}$ defined at the beginning of this section turns out to be a ubiquitous tool for the stability analysis of the periodic travelling waves $\bU=\ubU(x-\speed t)$ defined by \eqref{eq:EL}-\eqref{eq:ELham}. We shall repeatedly meet  its second variational derivative, $\Linvar=\Hess(\Ham+\speed \Impulse)[\ubU]$, which depends not only on $\speed$ but also on $(\blambda,\mu)$ through the profile $\ubU$
and whose spectrum undoubtedly plays a crucial role in the stability or instability of $\ubU$.
In addition, the value of $\funct^{(\speed,\blambda,\mu)}$ at $\ubU$, which we have denoted by $\Theta(\mu,\blambda,\speed)$, and the variations of $\Theta$ with respect to $(\speed,\blambda,\mu)$ show up in stability conditions from both the spectral and modulational points of view.

\subsection{Spectral point of view}\label{ss:spect}

A widely used approach to stability of equilibria consists in \emph{linearizing} about these equilibria. Even though periodic waves $\bU=\ubU(x-\speed t)$ are not genuine equilibria, they can be changed into \emph{stationary} solutions by making a change of frame. Indeed, in a frame moving with speed $\speed$, Eq.~\eqref{eq:absHamb} becomes
$$\partial_t\bU -\,\speed \partial_x\bU= \partial_x (\bB\,\Euler \Ham[\bU])\,,$$
or equivalently,
$$\partial_t\bU= \bB \partial_x (\,\Euler (\Ham\,+\,\speed \Impulse)[\bU])\,,$$
which admits $\ubU=\ubU(x)$ as special solutions. Linearizing about $\ubU$ we receive
the system
$$\partial_t\bU= \bB \partial_x (\,\Linvar \bU)\,,$$
where we recognize $\Linvar=\Hess(\Ham+\speed \Impulse)[\ubU]$.  Therefore, the linearized stability of $\ubU$ should be encoded by the spectrum of $\Lin=\sJ \Linvar$ with $\sJ=\bB \partial_x$. By definition, $\ubU$ will be said to be \emph{spectrally stable} if the operator $\Lin$ has no spectrum in the right-half plane. Note that, since $\sJ$ is skew-adjoint and $\Linvar$ is self-adjoint -~and both are real-valued~-,
possible eigenvalues of $\Lin$ arise as quadruplets $(\tau,\overline{\tau},-\tau,-\overline{\tau})$. 
This means that any eigenvalue outside the imaginary axis would imply instability. Furthermore,
according to  \cite[Theorem 3.1]{PegoWeinstein}, the number of eigenvalues of  $\Linvar$ in the left-half plane controls,  in some sense, the number of unstable eigenvalues of $\Lin$. Recalling that $\Linvar$ has at least one negative eigenvalue, there is room for (at least) one unstable eigenvalue of $\Lin$.

These considerations are rather loose actually, because the spectrum of a differential operator depends on the chosen functional framework. We may look at the differential operator $\Lin$ as an unbounded operator on $L^2(\R/\Xi\Z)$, in which case its spectrum is entirely made of isolated eigenvalues. These concern what is usually called \emph{co-periodic} spectral stability. We may widen the class of possible perturbations and consider $\Lin$ as an unbounded operator on $L^2(\R/n\Xi\Z)$ with $n$ any integer greater than one.  
Finally, we may consider `localized' perturbations by looking at 
$\Lin$ as an unbounded operator on $L^2(\R)$. As was shown by Gardner \cite{Gardner93}, the spectrum of 
$\Lin$ on $L^2(\R)$ is made of a collection of closed curves of so-called $\nu$-eigenvalues. For any $\nu \in \R/2\pi\Z$, a $\nu$-eigenvalue is an eigenvalue of the operator 
$\Lin^\nu:=\Lin(\partial_x+{i\nu}/{\Xi})$ on $L^2(\R/\Xi\Z)$. These definitions are motivated by the equivalence, which holds for all $\tau\in \C$,
$$(\Lin \bU=\tau \bU\,,\;\bU(\cdot+\Xi)=\ee^{i\nu}\bU
)\; \Leftrightarrow \; 
(\Lin^\nu \bU^\nu=\tau \bU^\nu\,,\;\bU^\nu(\cdot+\Xi)=\bU^\nu
)\,,$$
where we have introduced the additional notation 
$$\bU^\nu:x\mapsto \bU^\nu(x)=\ee^{-i\nu x/\Xi}\,\bU(x)\,.$$
All this is linked to the Floquet theory of ODEs with periodic coefficients, and we shall refer to $\nu$ as a \emph{Floquet exponent}.
Furthermore, there is a tool encoding all kinds of spectral stability,  with respect to either square integrable, or multiply-periodic, or just co-periodic perturbations. Indeed, under the assumption made earlier 
in \eqref{structure1}-\eqref{structure2}-\eqref{structure3}
that 
$$\Ham= \Ham(v,\bu,v_x)\quad \mbox{with}\quad\frac{\partial^2 \Ham}{\partial v_x^2}=\kappa(v)>0
\quad \mbox{and}\quad\nabla^2_{\bu}\Ham=\bT(v)>0\,,$$
the eigenvalue equation $\Lin \bU=\tau \bU$ is equivalent to a system of $(N+3)$ ODEs (because it involves three derivatives of $v$). If $\Fund(\cdot;\tau)$ denotes its fundamental solution, the existence of a nontrivial $\bU$ such that
$$\Lin \bU=\tau \bU\,,\;\bU(\cdot+\Xi)=\ee^{i\nu}\bU\,,$$
is equivalent to $\Evans(\tau,\nu)=0$, where
$$\Evans(\tau,\nu)\,:=\,\det(\Fund(\Xi;\tau)\,-\,\ee^{i\nu})\,.$$
This $\Evans=\Evans(\tau,\nu)$  has been called an Evans function. According to its definition, 
$\Evans(\tau,0)=0$ means that $\tau$ is an eigenvalue of $\Lin$ on $L^2(\R/\Xi\Z)$. In other words, if $D(\cdot;0)$ vanishes somewhere outside the imaginary axis, the wave $\ubU$ is unstable with respect to co-periodic perturbations.
Similarly, if for any $n\in \N^*$ there is a zero of $D(\cdot,2\pi/n)$ outside the imaginary axis, the wave $\ubU$ is unstable with respect to perturbations of period $n\Xi$. If for any $\nu \in \R/2\pi\Z$, $D(\cdot,\nu)$ has a zero outside 
the imaginary axis, then this zero is an eigenvalue of $\Lin$ on $L^\infty(\R)$, and also belongs to the spectrum of 
 $\Lin$ on $L^2(\R)$, 
 $$\displaystyle\sigma(\Lin)\,=\,\bigcup_{\nu\in\R/2\pi\Z}\sigma (\Lin^\nu)\,,$$
 which means that the wave $\ubU$ is unstable with respect to both bounded and square integrable perturbations.

Therefore, locating the zeroes of $\Evans(\cdot,\nu)$ when $\nu$ varies over $\R/2\pi\Z$ provides valuable information on the stability of the wave $\ubU$. The  `only' problem with $\Evans$ is that it is not known explicitly in general. If we are not to rely on numerical computations, we can only determine some of its asymptotic behaviors. This is often sufficient to prove instability results. A most elementary way concerns \emph{co-periodic instability}. Indeed, since the operator $\Linco$ is real-valued, the function $\Evans(\cdot,0)$ can be constructed so as to be real-valued too. In this case, finding a zero of $\Evans(\cdot,0)$ on $(0,+\infty)$ 
may just be 
a matter of applying the mean value theorem, once we know the behavior of $\Evans(\tau,0)$ for $|\tau|\ll 1$ and for $\tau\gg 1$, $\tau\in \R$. Another possibility is to detect \emph{side-band instability}, which occurs when a zero of 
$\Evans(\cdot,\nu)$ bifurcates from $0$ 
into the right half-plane 
for $|\nu|\ll 1$.

Let us mention that alternative approaches to locate unstable eigenvalues have been proposed that use, for instance, the Krein signature. See \cite{KollarMiller} for new insight on these distinct tools that are the Evans function and the Krein signature, and for the definition of an Evans-Krein function, which carries more information regarding the eigenvalue count than the original Evans function. However, the approach in  \cite{KollarMiller} does not apply here because our operator $\sJ$ is not onto. This is a recurrent difficulty with KdV-like PDEs.

\subsection{Modulational point of view} \label{ss:modul}
We consider an open set $\Omega$ of $(\mu, \blambda, \speed)$ and assume that we have a smooth mapping
$(\mu, \blambda, \speed)\in \Omega \mapsto (\ubU,\Xi)$ such that
\eqref{eq:EL}-\eqref{eq:ELham} hold true with $\ubU(0)=\ubU(\Xi)$, and $\uv_x(0)=\uv_x(\Xi)=0$.
We are interested in mild modulations of the wave $\bU=\ubU(x-\speed t)$, in which $(\mu, \blambda, \speed)$ will vary according to a slow time $T=\varepsilon t$ and on a small length $X=\varepsilon x$, with $\varepsilon\ll 1$.
The so-called modulated equations will consist of conservation laws in the $(X,T)$ variables for 
\begin{itemize}
\item the 
wave number, $\wn=1/\Xi$,
\item the mean value of the wave, $\meanU:= \wn \int_{0}^{\Xi} \ubU \,\dif x$,
\item the mean value of the impulse, $\meanI:= \wn \int_{0}^{\Xi} \Impulse(\ubU) \,\dif x$.
\end{itemize}

Before writing down these equations, let us see whether the mapping 
$(\mu, \blambda, \speed)\mapsto (\wn,\meanU,\meanI)$ has any chance to be a diffeomorphism. 
A `natural' condition for this to occur turns out to depend on the Hessian of 
\begin{equation}\label{eq:defTheta}
\Theta(\mu,\blambda,\speed)
:=\int_{0}^{\Xi} (\Ham(\ubU,\uv_x)+\speed \Impulse(\ubU) + \lambda\cdot \ubU +\mu)\,\dif x\,,
\end{equation}
as a function of its $(N+2)$ variables.
This is because $\Theta(\mu,\blambda,\speed)$ coincides with the \emph{action}
of the profile ODEs \eqref{eq:EL}, when viewed as a Hamiltonian system associated with the Hamiltonian $\Legendre \Lag$. Indeed, by \eqref{eq:ELham}, we have 
$$ \uv_x\frac{\partial \Ham}{\partial v_x}(\ubU,\uv_x)\,-\, \Ham [\ubU]\,-\,\speed \Impulse(\ubU) - \lambda\cdot \ubU =\mu\,,$$
hence by change of variable
$$\Theta(\mu,\blambda,\speed)\,=\,\oint \frac{\partial \Ham}{\partial v_x}(\ubU,\uv_x)\,\dif v\,,$$
where the symbol $\oint$ stands for the integral in the $(v,v_x)$-plane along the orbit described by $\uv$.

\begin{proposition}
\label{prop:param}
Assume that $\Ham= \Ham(\bU,v_x)$ is smooth, and that we have a smooth mapping
$$(\mu, \blambda, \speed)\in \Omega \mapsto (\ubU,\Xi)\;\mbox{s.t.
\eqref{eq:EL}-\eqref{eq:ELham} hold true, and }\; \ubU(0)=\ubU(\Xi), \;\uv_x(0)=\uv_x(\Xi)=0\,.$$
Then the function $\Theta$ defined in \eqref{eq:defTheta} is also smooth, and we have
\begin{equation}
\label{eq:param}
\frac{\partial \Theta}{\partial \mu}\,=\,\Xi\,,\quad\frac{\partial \Theta}{\partial \speed}\,=\,\int_{0}^{\Xi} \Impulse (\ubU)\,\dif x\,,\quad \nabla_{\blambda}  \Theta\,=\,\int_{0}^{\Xi} \ubU\, \dif x\,.
\end{equation}
\end{proposition}

\begin{proof}
This is a calculus exercise. Denoting for simplicity by $\emm$ the function 
$$\emm(\bU,v_x;\speed,\blambda,\mu):=\Lag(\bU,v_x;\speed,\blambda)\,+\,\mu\,=\, \Ham(\bU,v_x)+\speed \Impulse(\bU) + \lambda\cdot \bU +\mu\,,$$
if $a$ is any of the parameters $\speed,\lambda_\alpha,\mu$, and if we denote by a subscript derivation with respect to $a$,
we have 
$$\Theta_a\,=\,\begin{array}[t]{l}\displaystyle \int_{0}^{\Xi} \emm_a(\ubU,\uv_x;\speed,\blambda,\mu)\,\dif x\,+\,
\Xi_a\,\emm(\ubU(\Xi),\uv_x(\Xi);\speed,\blambda,\mu)\\  [5pt]
\displaystyle\,+\,\int_{0}^{\Xi} \Big(\ubU_a\cdot \nabla_{\bU} \emm(\ubU,\uv_x;\speed,\blambda,\mu)\,+\,\uv_{x,a}\,\frac{\partial \Ham}{\partial v_x}(\ubU,\uv_x)\Big)\Big)\dif x.\end{array}$$
The announced formulas rely on the observation that all but the first term in the right-hand side here above equal zero.
To show this, let us insist on the fact that, by \eqref{eq:ELham},
$$\emm(\ubU,\uv_x;\speed,\blambda,\mu)= \uv_x \,\frac{\partial \Ham}{\partial v_x}(\ubU,\uv_x)\,.$$
Since $\uv_x(\Xi)=0$, we thus readily see that $\emm(\ubU(\Xi),\uv_x(\Xi);\speed,\blambda,\mu)=0$. In order to deal with the last, integral term in $\Theta_a$, we observe that 
$\uv_{x,a}=\partial_x \uv_a$, and make an integration by parts, in which the boundary terms cancel out, again because
$\uv_x(0)=\uv_x(\Xi)=0$. This yields
$$\int_{0}^{\Xi} \Big(\ubU_a\cdot \nabla_{\bU} \emm(\ubU,\uv_x;\speed,\blambda,\mu)\,+\,\uv_{x,a}\,\frac{\partial \Ham}{\partial v_x}(\ubU,\uv_x)\Big)\Big)\dif x \,=\,
\int_{0}^{\Xi} \ubU_a\cdot \Euler \Lag[\ubU]\,\dif x\,,$$
which is equal to zero because of \eqref{eq:EL}.
\end{proof}

\begin{corollary}
\label{cor:param}
Under the assumptions of Proposition \ref{prop:param}, the mapping 
$(\mu, \blambda, \speed)\in \Omega \mapsto (\wn, \meanU,\meanI)$ is a diffeomorphism if and only if it is one-to-one and 
$$\det \Big(\Hess \Theta (\mu, \blambda, \speed)\Big)\,\neq\,0\,,\;\forall (\mu, \blambda, \speed)\in \Omega\,.$$
\end{corollary}

\begin{proof} The mapping 
$(\mu, \blambda, \speed)\in \Omega \mapsto (\wn=1/\Xi, \meanU=\wn \int_{0}^{\Xi} \ubU \dif x,\meanI=\wn \int_{0}^{\Xi} \Impulse(\ubU) \dif x)$ is clearly a diffeomorphism if and only if 
$$\textstyle(\mu, \blambda, \speed)\in \Omega \mapsto (\Xi, \int_{0}^{\Xi} \ubU \,\dif x,\int_{0}^{\Xi} \Impulse(\ubU) \,\dif x)$$ is so.
By  Proposition \ref{prop:param}, we have that
$$\textstyle(\Xi, \int_{0}^{\Xi} \uU_1\,\dif x\,,\ldots,\,\int_{0}^{\Xi} \uU_N\, \dif x\,,\,\int_{0}^{\Xi} \Impulse(\ubU) \,\dif x)^\transp=\nabla \Theta(\mu, \blambda, \speed)\,,$$ and the Jacobian matrix of $(\mu, \blambda, \speed)\mapsto \nabla \Theta(\mu, \blambda, \speed)$  is by definition the Hessian of $\Theta$.
\end{proof}

Let us assume that $(\mu, \blambda, \speed)\in \Omega \mapsto (\wn, \meanU,\meanI)$ is indeed a diffeomorphism. Then periodic wave profiles may be parametrized by $(\wn, \meanU,\meanI)$ instead of $(\mu, \blambda, \speed)$. In what follows, we make the dependence on  $(\wn, \meanU,\meanI)$ explicit by denoting such profiles by $\ubU^{(\wn,\meanU,\meanI)}$, which in addition we \emph{rescale} so that they all have the same period, say one.
Then each of them is associated with a travelling wave solution to \eqref{eq:absHam} by setting $$\bU(t,x)=\ubU^{(\wn,\meanU,\meanI)}(\wn x+\freq(\wn,\meanU,\meanI) t)\,,$$ of speed
$\speed=\speed(\wn,\meanU,\meanI)$, and {time frequency}  $\freq= \freq(\wn,\meanU,\meanI):=-\wn\,\speed(\wn,\meanU,\meanI)$.

We are interested in solutions to \eqref{eq:absHam} taking the form of slowly modulated wave trains
$$\bU(t,x)\ =\ \ubU^{(\wn,\meanU,\meanI)(\varepsilon t, \varepsilon x)}\big(\,\tfrac{1}{\varepsilon}\,{\phi({\varepsilon t},{\varepsilon x})}\,\big)\,+\,\mathcal{O}(\varepsilon)\,, $$
with $\phi=\phi(T,X)$ such that $\phi_X=\wn$ and $\phi_T=\freq$. 
(Note that when $(\wn,\meanU,\meanI)$ is independent of  $(T,X)$, we just recover exact, periodic travelling wave solutions.)
Whitham's averaged equations consist of conservation laws for 
$(\wn,\meanU,\meanI)=(\wn,\meanU,\meanI)(T,X)$ obtained by formal asymptotic expansions.
In fact, the equation on $\wn$ is just obtained by the Schwarz lemma applied to the phase $\phi$, \begin{equation}
\label{eq:Whithamk}
\partial_T k + \partial_X(ck)=0\,.
\end{equation}
The equations on $\meanU$ and $\meanI$ are derived by plugging the more precise ansatz 
$$\bU(t,x)=\bU^0(\varepsilon t, \varepsilon x, \phi(\varepsilon t, \varepsilon x)/\varepsilon)\,+\,
\varepsilon\,\bU^1(\varepsilon t, \varepsilon x, \phi(\varepsilon t, \varepsilon x)/\varepsilon,\varepsilon)\,+\,o(\varepsilon)\,,$$
in \eqref{eq:absHam} and \eqref{eq:impulsecl} respectively, assuming that $\bU^0$ and $\bU^1$ are $1$-periodic in their third variable $\theta$ (the rescaled phase).
The $O(1)$ terms vanish provided that $$\bU^0(T,X,\theta)=\ubU^{(\wn,\meanU,\meanI)(T,X)}(\theta)\,.$$
With this choice, the $O(\varepsilon)$ terms involving $\bU^1$ cancel out when averaging, and we receive the equations
\begin{equation}
\label{eq:WhithamU}
\partial_T \meanU= \bB \partial_X \langle \Euler \Ham_k[\ubU^{(\wn,\meanU,\meanI)}] \rangle\,,
\end{equation}
\begin{equation}
\label{eq:WhithamI}
\partial_T \meanI=  \partial_X \langle \ubU\cdot \Euler \Ham_k[\ubU^{(\wn,\meanU,\meanI)}] \,+\,\Legendre \Ham_k[\ubU^{(\wn,\meanU,\meanI)}]\rangle\,.
\end{equation}
Here above, we have used the shortcut $\Ham_k:=\Ham(\bU,k\bU_\theta)$, and the Euler operator $\Euler$ and Legendre transform $\Legendre$ act as operators on functions of the rescaled variable $\theta$.
Of course we may simplify and write
$\langle \Euler \Ham_k[\ubU] \rangle= \langle \nabla_{\bU} \Ham_k(\ubU,k\ubU_\theta) \rangle$ in \eqref{eq:WhithamU}.
However, this is not as nice a simplification as the reformulation of the averaged equations given below. 

\begin{proposition}\label{prop:Whitham-action}
Under the assumptions of Proposition \ref{prop:param}, the system of equations in \eqref{eq:Whithamk}-\eqref{eq:WhithamU}-\eqref{eq:WhithamI} equivalently reads, as far as smooth solutions are concerned,
\begin{equation}
\label{eq:Whitham}\left\{\begin{array}{lclclr}
\partial_T\Big(\dfrac{\partial \Theta}{\partial \mu}\Big) & + &  \speed\, \partial_X\Big(\dfrac{\partial \Theta}{\partial \mu}\Big) & - & \Big(\dfrac{\partial \Theta}{\partial \mu}\Big)\,\partial_X\speed & =0\,,\\ [10pt]
 \partial_T\big(\nabla_{\blambda}\Theta\big) & + &  \speed\, \partial_X\big(\nabla_{\blambda}\Theta\big) & + & \Big(\dfrac{\partial \Theta}{\partial \mu}\Big)\,\bB\,\partial_X\blambda & =0\,,\\ [10pt]
 \partial_T\Big(\dfrac{\partial \Theta}{\partial \speed}\Big) & + &  \speed\, \partial_X\Big(\dfrac{\partial \Theta}{\partial \speed}\Big) & - & \Big(\dfrac{\partial \Theta}{\partial \mu}\Big)\,\partial_X\mu & =0\,.\\
\end{array}\right.
\end{equation}
or in quasilinear form,
\begin{equation}
\label{eq:Whithamql}
\bSigma\, \partial_T\bW + (\speed\, \bSigma\,+\,\Theta_\mu\bS)\partial_X\bW =0
\end{equation} with
$\bW^\transp:= (\mu,\blambda^\transp,\speed)$, $\bSigma:=\Hess\Theta$,  $\Theta_\mu= \dfrac{\partial \Theta}{\partial \mu}$ (at constant $\blambda$, $\speed$),
$$ \bS\,:=\,\left(\begin{array}{r|ccc|r} 0 & 0 & \cdots &  0 & -1 \\ \hline
0 &  & &  & 0 \\  [-6pt] 
\vdots & & \bB  & & \vdots \\ 
0 &  & &  & 0 \\ \hline
- 1 & 0 & \cdots & 0 & 0 \end{array}\right) \,.$$
\end{proposition}

\begin{proof}
Recalling that 
$$\dfrac{\partial \Theta}{\partial \mu}\,=\,\Xi\,=\,1/\wn\,,$$
and multiplying Eq.~\eqref{eq:Whithamk} by $-\Xi^2$, we readily obtain the first equation in \eqref{eq:Whitham}.
The other ones require a little more manipulations.
Regarding \eqref{eq:WhithamU}, we use that
$$\meanU= \wn \nabla_\blambda \Theta\,,$$
that by the profile equation \eqref{eq:EL} (after rescaling),
$$\Euler \Ham_k[\ubU^{(\wn,\meanU,\meanI)}]\,=\,-\speed\, \bB^{-1} \ubU^{(\wn,\meanU,\meanI)} \,-\,\blambda\,,$$
hence 
$$\bB\,\langle\Euler \Ham_k[\ubU^{(\wn,\meanU,\meanI)}]\rangle\,=\,-\speed\,\meanU \,-\,\bB\blambda\,,$$
and we eliminate the factor $k$ by using again \eqref{eq:Whithamk}. We proceed in a similar manner for 
\eqref{eq:WhithamI}, using that 
$$\meanI=\wn \;\Big(\dfrac{\partial \Theta}{\partial \speed}\Big)\,,$$
and that by the profile equations in \eqref{eq:EL}-\eqref{eq:ELham},
$$\ubU\cdot \Euler \Ham_k[\ubU^{(\wn,\meanU,\meanI)}] \,+\,\Legendre \Ham_k[\ubU^{(\wn,\meanU,\meanI)}]\,=\,
\mu\,-\,\speed\, \Impulse[\ubU^{(\wn,\meanU,\meanI)}]\,.$$
\end{proof}

\begin{remark} Would $\bSigma=\Hess \Theta$ be positive definite, \eqref{eq:Whitham} would automatically belong to the class of \emph{symmetrizable hyperbolic} systems, in view of its quasilinear form of \eqref{eq:Whithamql}.
However, as we shall see in Section \ref{s:necessary} (Theorem \ref{thm:unstab}), the definiteness of $\Hess \Theta$ is often incompatible with co-periodic stability. In other words, despite the nice, `symmetric' form of
the modulated equations \eqref{eq:Whitham},
their well-posedness is far from being automatic, especially in case of co-periodic stability.  The simultaneous occurrence of modulational stability and co-periodic stability remains possible though. This is in contrast with the framework of `quasi-gradient systems' considered in  \cite{PoganScheelZumbrun}, for which it has been shown that co-periodic and modulational stability are indeed incompatible.
\end{remark}

\section{Necessary conditions for stability}\label{s:necessary}

\subsection{Co-periodic instability criteria}
\begin{theorem}\label{thm:unstab} Under 
the structural conditions in \eqref{structure1}-\eqref{structure2}-\eqref{structure3}, and 
the assumptions of Proposition~\ref{prop:param},
\begin{itemize}
\item for $N=1$, if 
$\det (\Hess\Theta)>0$ then the wave is spectrally unstable with respect to co-periodic perturbations;
\item for $N=2$, if $\det (\Hess\Theta)<0$ 
then the wave is spectrally unstable with respect to co-periodic perturbations.
\end{itemize}
\end{theorem}
The first point  is a slight generalization ~-~with variable $\kappa(v)$~-~ of what was shown by 
Bronski and Johnson \cite{BronskiJohnson}.
The second point has been shown in \cite{BR} by means of an Evans function computation.

In both cases, the detected instability corresponds to a real positive unstable eigenvalue. Indeed, the sign criteria 
here above stem from a mod 2 count of such eigenvalues.

\subsection{Modulational instability implies side-band instability}

A necessary condition for spectral stability is modulational stability. This was shown by Serre \cite{Serre}, and by Oh and Zumbrun \cite{OhZumbrun10} for viscous periodic waves. In our framework, we have the following 
\begin{theorem}\label{thm:side-band} Assume that $\ubU$ is a periodic travelling wave profile, that the set of nearby profiles  is, up to translations, an $(N+2)$-dimensional manifold parametrized by $(\mu,\speed,\blambda)$, and that the generalized kernel of $\Lin$ in the space of $\Xi$-periodic functions is of dimension $N+2$. Then the system of modulated equations in \eqref{eq:Whithamk}-\eqref{eq:WhithamU}-\eqref{eq:WhithamI}, or equivalently \eqref{eq:Whitham}, is indeed an evolution system (in other words, $\Hess\Theta$ is nonsingular), and if it admits a nonreal characteristic speed then for any small enough Floquet exponent $\nu$, the operator $\Lin^\nu$ admits a (small) unstable eigenvalue.
\end{theorem}

This is a concatenation of results shown in \cite{BNR}.

If $\Xi=\Theta_\mu\neq 0$ (which we have implicitly assumed up to now), the hyperbolicity of \eqref{eq:Whithamql} is equivalent, by change of frame and rescaling, to that of 
$$\bSigma\, \partial_T\bW \,+\,\bS\,\partial_X\bW =0\,.$$ Assuming that $\bSigma=\Hess \Theta$ is nonsingular and noting that $\bS$ is always nonsingular (because we have assumed that $\bB$ is so), we thus see that the local well-posedness of the  averaged equations in \eqref{eq:Whitham} is equivalent to the fact that $\bS^{-1} \bSigma$ is diagonalizable on $\R$. 
Theorem \ref{thm:side-band} here above shows that spectral stability implies at least that the eigenvalues of $\bS^{-1} \bSigma$ are real.

\paragraph{Case $N=1$} (KdV). We have $\bS^{-1}=\bS\in \R^{3\times3}$, and 
$$\bS^{-1} \bSigma\,=\,\left(\begin{array}{ccc}-\Theta_{\speed\mu} & -\Theta_{\speed\lambda} & -\Theta_{\speed\speed}\\
\Theta_{\lambda\mu} & \Theta_{\lambda\lambda} & \Theta_{\lambda\speed}\\
-\Theta_{\mu\mu} & -\Theta_{\mu\lambda} & -\Theta_{\mu\speed}\end{array}\right)\,.$$
Then a necessary criterion for spectral stability is that the discriminant of the characteristic polynomial of this matrix be nonnegative. This criterion depends only on the second-order derivatives of the action $\Theta$.

For the case $N=2$ (EK), a $4\times4$ matrix is be to analyzed, and a similar necessary criterion can be explicitly obtained in terms of second-order derivatives of the action $\Theta$.

\paragraph{Small-amplitude limit.} 
A necessary condition for modulational stability of small-amplitude waves is two-fold and requires:
1) the hyperbolicity of the reduced system obtained in the zero-dispersion limit;
2) the so-called Benjamin--Feir--Lighthill criterion.
We refer to \cite{BNR} for more details. Observe in particular that the first condition is trivial in the case $N=1$ (because all scalar, first order conservation laws are hyperbolic), and has hardly ever been noticed. In the case $N=2$, and in particular for the Euler--Korteweg system, it requires that the Euler system 
be 
hyperbolic at the mean value of the wave. This is a nontrivial condition, which rules out some of the periodic waves in the Euler--Korteweg system when it is endowed with, for instance, the van der Waals pressure law. As to the Benjamin--Feir--Lighthill criterion, it is famous for characterizing the unstable Stokes waves.

\section{Sufficient conditions for stability}\label{s:sufficient}
\subsection{Grillakis--Shatah--Strauss criteria}\label{ss:GSS}
We assume as before that $\Ham=\Ham(v,\bu,v_x)$, and use the short notation $\HH^s$ for 
$H^{s}(\R/\Xi\Z)\times (L^2
(\R/\Xi\Z))^{N-1}$. Observe in particular that the functional $$\funct^{(\speed,\blambda,\mu)}: \bU\mapsto \int_{0}^{\Xi} (\Ham(\bU,\bU_x)+\speed\Impulse(\bU)+\blambda\cdot \bU+\mu)\,\dif x$$
is well-defined on $\HH^1$.
What we call a Grillakis--Shatah--Strauss (GSS) criterion is a set of inequalities regarding the second derivatives of $\Theta$ ensuring that the functional $\funct^{(\speed,\blambda,\mu)}$ admits a local minimum at 
$\ubU$ (and any one of its translates) on $\HH^1\cap {\mathscr C}$ with  
$${\mathscr C}=\{\bU\in \HH^0\,;\;  \textstyle \int_{0}^{\Xi}\Impulse(\bU)\,\dif x=\int_{0}^{\Xi}\Impulse(\ubU)\,\dif x\,,
\int_{0}^{\Xi} \bU\,\dif x=\int_{0}^{\Xi} \ubU\, \dif x\}\,.$$

By a  Taylor expansion argument, seeking a GSS criterion amounts to finding conditions under which the operator $\Linvar=\Hess(\Ham +\speed \Impulse)[\ubU]$ is nonnegative on $\HH^2\cap T_{\ubU} {\mathscr C}$ with
$$T_{\ubU} {\mathscr C}:=\{\bU\in \HH^0\,;\;  \textstyle \int_{0}^{\Xi}\bU\cdot\nabla_{\bU}\Impulse(\ubU) \,\dif x =0\,,\;\int_{0}^{\Xi} \bU\, \dif x=0\}\,.$$
Even though it might not be clear at once that such criteria exist, they do. As we have recalled above (in \S\ref{ss:var}), for solitary waves a now well-known GSS criterion \cite{GSS}  is $\momB_{\speed\speed}>0$, where $\momB$ is to solitary waves what our $\Theta$ is to periodic waves. 
Behind this criterion is a rather general result, pointed out at various places and shown in most generality by 
Pogan, Scheel, and Zumbrun \cite{PoganScheelZumbrun}, which makes the connection between the negative signatures of the unconstrained version of the Hessian of the functional we are trying to minimize, of its constrained version, and of the  
Jacobian matrix of the values of the constraints in  terms of the Lagrange multipliers. More explicitly, 
in our framework with our notations, and under some `generic' assumptions, the negative signature $\neg(\Linvar)$ of the operator $\Linvar$ is found to be equal to the negative signature $\neg(\Linvar_{|T_{\ubU}{\mathscr C}})$
of its restriction to $T_{\ubU}{\mathscr C}$ plus the negative signature $\neg(-\bC)$ of $-\bC$, where $\bC$ is the Jacobian matrix of the values of the constraints, 
$\int_{0}^{\Xi} \ubU$, $\int_{0}^{\Xi}\Impulse(\ubU)$, in terms of the Lagrange multipliers $(\blambda,\speed)$ when the period $\Xi$ is fixed. The counterpart of this matrix $\bC$ for solitary waves is just the scalar  $\momB_{\speed\speed}$, in which case 
we readily see that $\momB_{\speed\speed}>0$ is equivalent to $\neg(-\bC)=1$.
We now give a version of the Pogan--Scheel--Zumbrun theorem adapted to our framework and notations for periodic waves.

\begin{theorem}
\label{thm:PSZ}
Under the hypotheses of Proposition \ref{prop:param}, 
we assume moreover that $\Xi_\mu\neq 0$, and that $$\bC:= \check\nabla^2 \Theta \,-\,\frac{\check\nabla \Xi\otimes \check\nabla \Xi}{\Xi_\mu}$$ takes nonsingular values, with $\Theta$ defined as in \eqref{eq:defTheta} by
$$\Theta(\mu,\blambda,\speed):= \int_{0}^{\Xi} (\Ham(\ubU,\ubU_x)+\speed \Impulse(\ubU) + \blambda\cdot \ubU + \mu)\,\dif x\,,$$
and $\check\nabla$ being a shortcut for the gradient with respect to $(\blambda,\speed)$ at fixed $\mu$.
Then, denoting $\Linvar:=\Hess(\Ham+\speed \Impulse)[\ubU]$, we have
$$\neg(\Linvar) = \neg(\Linvar_{|T_{\ubU}{\mathscr C}}) \,+\,\neg(-\bC)\,.$$

\end{theorem}

\begin{proof} By assumption, the period $\Xi$ of a given profile $\ubU$ is a smooth function of the $N+2$ parameters $(\mu,\blambda,\speed)$.
The fact that $\Xi_\mu\neq 0$ implies by the implicit function theorem that $\mu$  can be viewed as a smooth function $\mu=\mu(\Xi,\blambda,\speed)$, and that 
\begin{equation}\label{eq:dermu}
\frac{\partial \mu}{\partial \lambda_\alpha}= -\,\frac{\Xi_{\lambda_\alpha}}{\Xi_\mu}\,,\;\frac{\partial \mu}{\partial \speed}= -\,\frac{\Xi_{\speed}}{\Xi_\mu}\,.
\end{equation}
Since we are interested in the signature of $\Linvar$ on $\HH^0=(L^2(\R/\Xi\Z))^N$, we shall mostly concentrate on travelling profiles of \emph{fixed} period $\Xi$, which are solution to \eqref{eq:EL}-\eqref{eq:ELham} with $\mu=\mu(\Xi,\blambda,\speed)$.
For such profiles, let us denote by $\bq$ the constraints mapping
$$\bq:(\blambda,\speed)\mapsto \textstyle(\int_{0}^\Xi \ubU\,\dif x, \int_{0}^\Xi \Impulse(\ubU)\,\dif x)\,,$$
and $q_\alpha$ its components, $\alpha\in \{1,\ldots,N+1\}$,
$$q_\alpha(\blambda,\speed):=\int_{0}^\Xi \uU_\alpha\,\dif x\,,\;\alpha\in \{1,\ldots,N\}\,,\quad q_{N+1}(\blambda,\speed):=\int_{0}^\Xi \Impulse(\ubU)\,\dif x\,.$$
From Eqs \eqref{eq:param} in Proposition \ref{prop:param} and Eqs in \eqref{eq:dermu}, we infer that for $\alpha\,,\beta\leq N$,
$$\frac{\partial q_{\beta}}{\partial \lambda_\alpha}= \Theta_{\lambda_\alpha \lambda_\beta}\, -\,\frac{\Xi_{\lambda_\alpha}\Xi_{\lambda_\beta}}{\Xi_\mu}\,,\;\frac{\partial q_{\beta}}{\partial c}= \Theta_{c \lambda_\beta}\, -\,\frac{\Xi_{c}\Xi_{\lambda_\beta}}{\Xi_\mu}\,\,,$$
$$\frac{\partial q_{N+1}}{\partial \lambda_\alpha}= \Theta_{\lambda_\alpha c}\, -\,\frac{\Xi_{\lambda_\alpha}\Xi_{c}}{\Xi_\mu}\,,\;\frac{\partial q_{N+1}}{\partial c}= \Theta_{c c}\, -\,\frac{\Xi_{c}\Xi_{c}}{\Xi_\mu}\,\,.$$
In other words, the Jacobian matrix of $\bq$ is indeed 
$$\bC= \check\nabla^2 \Theta \,-\,\frac{\check\nabla \Xi\otimes \check\nabla \Xi}{\Xi_\mu}=\check\nabla^2 \Theta \,-\,\frac{\check\nabla \Theta_\mu\otimes \check\nabla  \Theta_\mu}{ \Theta_{\mu\mu}}\,.$$
Now, by differentiating \eqref{eq:EL} with respect to $\mu$, $\blambda$ or $\speed$, we see that
\begin{equation}\label{eq:Linvarid}
\Linvar \left(\ubU_{\lambda_\alpha}\,-\,\tfrac{\Xi_{\lambda_\alpha}}{\Xi_\mu}\,\ubU_\mu\right)\,=\,-\,
\ee_{\alpha}\,,\;\Linvar \left(\ubU_{\speed}\,-\,\tfrac{\Xi_{\speed}}{\Xi_\mu}\,\ubU_\mu\right)\,=\,-\,\nabla_{\bU} \Impulse(\ubU)\,,
\end{equation}
where $\ee_{\alpha}$ denotes the $\alpha$-th vector of the `canonical' basis of $\R^N$, hence the alternative expression
for $\alpha\,,\beta\leq N$, 
\begin{equation}\label{eq:altC} 
\begin{array}{rcl}
\bC_{\alpha,\beta}&=&-\left\langle \left(\ubU_{\lambda_\beta}\,-\,\tfrac{\Xi_{\lambda_\beta}}{\Xi_\mu}\,\ubU_\mu\right)\cdot \Linvar \left(\ubU_{\lambda_\alpha}\,-\,\tfrac{\Xi_{\lambda_\alpha}}{\Xi_\mu}\,\ubU_\mu\right)\right\rangle\\[0.5em]
\bC_{\alpha,N+1}&=&-\left\langle \left(\ubU_{\speed}\,-\,\tfrac{\Xi_{\speed}}{\Xi_\mu}\,\ubU_\mu\right)\cdot \Linvar \left(\ubU_{\lambda_\alpha}\,-\,\tfrac{\Xi_{\lambda_\alpha}}{\Xi_\mu}\,\ubU_\mu\right)\right\rangle\\[0.5em]
\bC_{N+1,\beta}&=&-\left\langle \left(\ubU_{\lambda_\beta}\,-\,\tfrac{\Xi_{\lambda_\beta}}{\Xi_\mu}\,\ubU_\mu\right)\cdot \Linvar \left(\ubU_{\speed}\,-\,\tfrac{\Xi_{\speed}}{\Xi_\mu}\,\ubU_\mu\right)\right\rangle\\[0.5em]
\bC_{N+1,N+1}&=&-\left\langle \left(\ubU_{\speed}\,-\,\tfrac{\Xi_{\speed}}{\Xi_\mu}\,\ubU_\mu\right)\cdot \Linvar \left(\ubU_{\speed}\,-\,\tfrac{\Xi_{\speed}}{\Xi_\mu}\,\ubU_\mu\right)\right\rangle
\end{array}
\end{equation}
where $\langle\ \cdot\ \rangle$ denotes the inner product in $L^2(\R/\Xi\Z;\R^N)$.
This implies, if $\bC$ is nonsingular, that
$$\HH^0=\mbox{Span}(\ubU_{\lambda_1}\,-\,\tfrac{\Xi_{\lambda_1}}{\Xi_\mu}\,\ubU_\mu,\ldots,\ubU_{\lambda_N}\,-\,\tfrac{\Xi_{\lambda_N}}{\Xi_\mu}\,\ubU_\mu,\ubU_{\speed}\,-\,\tfrac{\Xi_{\speed}}{\Xi_\mu}\,\ubU_\mu)\,\oplus \,T_{\ubU} {\mathscr C}\,,$$
$$T_{\ubU} {\mathscr C}=\{\bU\in \HH^0\,;\;  
\langle \bU\cdot \nabla_{\bU}\Impulse(\ubU)\rangle\,=\,0\,,\;\langle \bU\rangle =0\}\,.$$
As a matter of fact, Equations in \eqref{eq:Linvarid}-\eqref{eq:altC} imply that for any $\bV\in \HH^0$, there is one and only one $(a_1,\ldots, a_{N+1},\bU)\in \R^{N+1}\times T_{\ubU} {\mathscr C}$ such that
$$\bV=a_1\left(\ubU_{\lambda_1}\,-\,\tfrac{\Xi_{\lambda_1}}{\Xi_\mu}\,\ubU_\mu\right)\,+\,\cdots\,\,+\,a_N\left(\ubU_{\lambda_N}\,-\,\tfrac{\Xi_{\lambda_N}}{\Xi_\mu}\,\ubU_\mu\right)\,+\,a_{N+1}\left(\ubU_{\speed}\,-\,\tfrac{\Xi_{\speed}}{\Xi_\mu}\,\ubU_\mu\right)\,+\,\bU\,,$$ 
which can be computed by solving the $(N+1)\times (N+1)$ system
$$\left(\begin{array}{c}
\int_{0}^\Xi V_1\,\dif x \\
\vdots \\
\int_{0}^\Xi V_N\,\dif x \\
\int_{0}^{\Xi}\bV\cdot\nabla_{\bU}\Impulse(\ubU) \,\dif x  \end{array}\right)\,=\,\bC\, \left(\begin{array}{l}
a_1 \\
\vdots \\
a_N \\
a_{N+1}  \end{array}\right)\,.
$$
In order to conclude, let us denote by $\bPi_0$ the orthogonal projection onto the space 
$$\mbox{Span}(\ubU_{\lambda_1}\,-\,\tfrac{\Xi_{\lambda_1}}{\Xi_\mu}\,\ubU_\mu,\ldots,\ubU_{\lambda_N}\,-\,\tfrac{\Xi_{\lambda_N}}{\Xi_\mu}\,\ubU_\mu,\ubU_{\speed}\,-\,\tfrac{\Xi_{\speed}}{\Xi_\mu}\,\ubU_\mu)\,,$$
and by $\bPi_1$ the orthogonal projection onto $T_{\ubU} {\mathscr C}$. We readily see that for all $\bU\in T_{\ubU} {\mathscr C}$ and 
$\bV_0\in \mbox{Span}(\ubU_{\lambda_1}\,-\,\tfrac{\Xi_{\lambda_1}}{\Xi_\mu}\,\ubU_\mu,\ldots,\ubU_{\lambda_N}\,-\,\tfrac{\Xi_{\lambda_N}}{\Xi_\mu}\,\ubU_\mu,\ubU_{\speed}\,-\,\tfrac{\Xi_{\speed}}{\Xi_\mu}\,\ubU_\mu)$, 
$$\langle \bU \cdot \Linvar \bV_0\rangle=0\,,$$
hence 
$$\bPi_1 \Linvar \bPi_0 =0\,,\;\bPi_0 \Linvar \bPi_1 =0\,.$$
Therefore, for all $\bV\in {\mathscr D}(\Linvar)=\HH^2$,
$$\langle \bV \cdot \Linvar\bV\rangle= \langle \bV \cdot \bPi_0 \Linvar \bPi_0 \bV\rangle\,+\,\langle \bV \cdot \bPi_1 \Linvar \bPi_1 \bV\rangle\,.$$
From this relation we see that the negative signature of $\Linvar$ is the sum of those of 
$\bPi_0 \Linvar \bPi_0$ and $ \bPi_1 \Linvar \bPi_1 $. The latter is the negative signature of $\Linvar_{|T_{\ubU}{\mathscr C}}$, by definition of the projection $\bPi_1$, while the former coincides with the negative signature of $-\bC$ by definition of the projection $\bPi_0$ and
by the expression of $\bC$ in \eqref{eq:altC}.
\end{proof}

Note that ${\mathscr C}$ is a codimension $(N+1)$ manifold of $\HH^0$. As a matter of fact, the constraints defining ${\mathscr C}$ are `full rank', in the sense that 
for all $(\bm,q)\in \R^{N+1}$, there exists $\bU\in \HH^0$ such that
$$\int_{0}^\Xi \bU\,\dif x \,=\, \bm\,,\; \int_{0}^\Xi \bU\cdot\nabla_{\bU} \Impulse(\ubU)\,\dif x\,=\,q\,.$$
Recalling that $\nabla_{\bU} \Impulse(\ubU)=\bB^{-1}\ubU$, we may take for instance $\bU=\frac{\bm}{\Xi}\,+\,a\,\bB^{-1} \ubU_{xx}$ with 
$$a=\frac{\bm\cdot \bB^{-1} \ubM -q}{\int_{0}^{\Xi} \|\bB^{-1}\ubU_{x}\|^2\,\dif x}\,.$$

\begin{corollary} 
If the negative signatures of the operator $\Linvar$ and of the matrix $\bC$ defined in Theorem \ref{thm:PSZ} are equal, then the periodic travelling wave $(x,t)\mapsto \ubU(x-\speed t)$ is (conditionally) orbitally stable to co-periodic perturbations.
\end{corollary}

\begin{proof}
From Theorem \ref{thm:PSZ} we infer that the negative signature of 
$\Linvar_{|T_{\ubU}{\mathscr C}}$ 
is zero. In other words, the functional $\funct^{(\speed,\blambda,\mu)}$ does have a local minimum at $\ubU$ on $\HH^1\cap {\mathscr C}$. (This follows from a Taylor expansion and the density of ${\mathscr D}(\Linvar)$ in $\HH^1$.) The fact that it is not a strict minimum can be coped with  by `factoring out' the translation-invariance problem in the usual way. Namely, by the implicit function theorem, there exists a tubular neighborhood ${\mathscr N}$ in $L^2(\R/\Xi\Z)$ of $\ubU$ and all its translates $\ubU(\cdot+\xi)$ for $\xi\in \R$, and a smooth mapping $s: {\mathscr N}\to \R$ such that for all $\bU\in {\mathscr N}$, 
$\bU(\cdot-s(\bU))-\ubU$ is orthogonal to $\ubU_x$. As a consequence, up to diminishing ${\mathscr N}$, we can find an $\alpha>0$ so that for all $\bU\in {\mathscr N}\cap \HH^1\cap {\mathscr C}$,
$$\funct^{(\speed,\blambda,\mu)}[\bU]-\funct^{(\speed,\blambda,\mu)}[\ubU]=\int_{0}^{\Xi} (\Ham(\bU,\bU_x)- \Ham(\ubU,\ubU_x))\,\dif x\,\geq \alpha\,\|\bU-\ubU(\cdot + s(\bU))\|^2_{\HH^1}\,.$$
This enables us to show the following, conditional stability result. If $\HH\subset \HH^1$ is such that the Cauchy problem associated with \eqref{eq:absHam} is locally well-posed in $\HH$, if we denote by $T(\bU_0)$ the maximal time of existence of the solution $\bU$ of \eqref{eq:absHam} in $\HH$ with initial data $\bU_0\in \HH$,
$$
\forall \varepsilon>0,\;\exists \delta>0\,;\;%&
\forall \bU_0\in \HH\,;\;
\|\bU_0\,-\,\ubU\|_{\HH^1}\,\leq \delta\; \Rightarrow \begin{array}[t]{l} \forall t\in [0,T(\bU_0))\,,\\ [5pt]
\displaystyle \inf_{s\in \R}\,\|\bU(t,\cdot)\,-\,\ubU(\cdot+s)\|_{\HH^1}\,\leq\,\varepsilon\,.\end{array}
$$
The proof works by contradiction, as in \cite{GSS,BSS}, even though an alternative, direct proof as in \cite{Johnson} is also possible. Assume there exist $\varepsilon>0$, and a sequence of initial data $\bU_{0,n}\in \HH$ such that 
$\inf_{s\in \R}\,\|\bU_{0,n}\,-\,\ubU(\cdot+s)\|_{\HH^1}$ goes to zero while 
$$\sup_{t\in[0,T(\bU_0))}\inf_{s\in \R}\,\|\bU_n(t,\cdot)\,-\,\ubU(\cdot+s)\|_{\HH^1}\,>\, \varepsilon\,.$$
Without loss of generality, we can assume that the tubular neighborhood of $\ubU$ of radius $2\varepsilon$ is contained in ${\mathscr N}$.
We choose $t_n$ to be the least value such that 
$$\inf_{s\in \R}\,\|\bU_n(t_n,\cdot)\,-\,\ubU(\cdot+s)\|_{\HH^1}\,=\, \varepsilon\,.$$
By invariance of $\int _0^\Xi \Ham[\bU]\,\dif x$, $\int _0^\Xi \Impulse[\bU]\,\dif x$, and 
$\int _0^\Xi \bU\,\dif x$, with respect to time evolution and spatial translations, we have
$$\int_0^\Xi \Ham[\bU_n(t_n)]\,\dif x\,=\,\int_0^\Xi \Ham[\bU_{0,n}]\,\dif x\,\to \,\int_0^\Xi \Ham[\ubU]\,\dif x\,,$$
$$\int_0^\Xi \Impulse(\bU_n(t_n))\,\dif x\,=\,\int_0^\Xi \Impulse(\bU_{0,n})\,\dif x\,\to \,\int_0^\Xi \Impulse(\ubU)\,\dif x\,,$$
$$\int_0^\Xi \bU_n(t_n)\,\dif x\,=\,\int_0^\Xi \bU_{0,n}\,\dif x\,\to \,\int_0^\Xi \ubU\,\dif x\,.$$
This implies, by using the full-rank property mentioned above and the submersion theorem that
we can pick for all $n$ some $\bV_n\in {\mathscr N}\cap\HH^1\cap {\mathscr C}$ such that
$\|\bV_n-\bU_n(t_n)\|_{\HH^1}$ goes to zero, as well as 
$$ \int_0^\Xi (\Ham[\bV_n]-\Ham[\ubU])\,\dif x\,\to\,0\,.$$
Therefore, 
$$\|\bV_n(\cdot)-\ubU(\cdot+s(\bV_n))\|^2_{\HH^1}\leq \,\frac{1}{\alpha}\,\int_0^\Xi (\Ham[\bV_n]-\Ham[\ubU])\,\dif x\,\to\,0\,,$$
hence 
$$\|\bU_n(t_n,\cdot)-\ubU(\cdot+s(\bV_n))\|\,\to\,0$$
by the triangular inequality. This is a contradiction of the definition of $t_n$.
\end{proof}

\begin{remark}\label{rem:KdV}
For (KdV), a case in which $N=1$, it has been shown by Johnson \cite{Johnson} that a periodic wave is orbitally stable to co-periodic perturbations under the two conditions
\begin{equation}\label{eq:stabJ}\Theta_{\mu\mu}>0\,,\qquad\det (\Hess\Theta)<0\,.
\end{equation}
It is not difficult to see that these assumptions imply that the constraints matrix $\bC$ has signature $(-,+)$. Indeed, for $N=1$ we have 
$${\bC}=\frac{1}{\Theta_{\mu\mu}}\,\left(\begin{array}{cc}\Theta_{\mu \mu} \Theta_{\lambda\lambda}-\Theta_{\mu \lambda} \Theta_{\lambda\mu} & \Theta_{\mu \mu} \Theta_{\speed\lambda}-\Theta_{\mu \speed} \Theta_{\lambda\mu}\\
\Theta_{\mu \mu} \Theta_{\lambda\speed}-\Theta_{\mu \lambda} \Theta_{\speed\mu} &  \Theta_{\mu \mu} \Theta_{\speed\speed}-\Theta_{\mu\speed} \Theta_{\speed\mu}\end{array}\right) \,,$$
and a bit of algebra shows that
$$ \Theta_{\mu\mu}\;\det \bC\,=\, \det (\Hess\Theta)\,,$$
so that if \eqref{eq:stabJ} hold true, $\det (-\bC)=\det \bC<0$ thus $\neg(-\bC)=1$. The result then follows from \cite[Lemma~4.2]{Johnson}, which  proves that $\Theta_{\mu\mu}>0$ implies $\neg(\Linvar)=1$.
\end{remark}

\section*{Appendix}

\paragraph{Table of examples}$\;$

\vspace{2mm}
\noindent
\begin{tabular}{|c|c|c|c|c|c|%c|c|c|
}\hline
&&&&&%&&&
\\
%$\begin{array}{c} \quad\mbox{\small{Quantity}} \\ \diagdown\\ \mbox{\small{Equation} }\quad\end{array}$ 
& $N$ & $\bU$ & $\sJ$ & $\Ham$ &   $\Impulse$ %& $\speed$ &  &  
\\  [10pt] \hline
&&&&&%&&&
\\
({KdV}) & 1 & v & $\partial_x$ & $\frac{1}{2}v_x^2+f(v)$ & $\tfrac{1}{2}\,v^2$ %& $\speed$ & &  
\\ [10pt] \hline
&&&&&%&&&
\\
({EKL}) & 2 & $\left(\begin{array}{c}\vol \\ \vitsL \end{array}\right)$& 
$\left(\begin{array}{cc}0 &  \partial_y \\
 \partial_y & 0\end{array}\right)$ & $\frac{1}{2} \vits^2\,+\,\en(\vol,\vol_y)$ & $\vol \vitsL$ %& $-j$  & & 
  \\ [10pt] \hline
&&&&&%&&&
\\
({EKE}) & 2 & $\left(\begin{array}{c}\rho \\ \vits \end{array}\right)$ & 
$-\left(\begin{array}{cc}0 &  \partial_x \\
 \partial_x & 0 \end{array}\right)$ & $\frac{1}{2}\rho \vits^2\,+\,\En(\rho,\rho_x)$ & $-\rho \vits$ %& $\sigma$  & &  
 \\ [10pt] \hline
&&&&&%&&&
\\
(B) & 2 & $\left(\begin{array}{c}\chi  \\ \chi_t \end{array}\right)$ & $\left(\begin{array}{cc}0 &   1 \\
 -1 &  0\end{array}\right)$ & $\frac{1}{2} \chi_t^2\,+\,W(\chi_x)\,\pm\,\frac{1}{2} \chi_{xx}^2$ & $ \chi_t\chi_x$%& $c$ & & 
 \\ [10pt] \hline
&&&&&%&&&
\\
(NLW) & 2 & $\left(\begin{array}{c}\chi \\   \chi_t \end{array}\right)$ & $\left(\begin{array}{cc}0 &   1 \\
 -1 &  0\end{array}\right)$ & $\frac{1}{2} \chi_t^2\,+\,\frac{1}{2} \chi_x^2\,+\,V(\chi)$ &$ \chi_t\chi_x$%& $c$ & & 
 \\ [10pt] \hline
&&&&&%&&&
\\
(NLS) & 2 & $\left(\begin{array}{c}\Real\psi \\ \Imag\psi \end{array}\right)$ & $\left(\begin{array}{cc}0 & 1 \\-1 & 0\end{array}\right)$   & $ \frac{1}{2} |\psi_x|^2\,+\,F(|\psi|^2)$ & $-\Imag(\overline{\psi} \, \psi_x )$ %& $c$ & & 
\\  [10pt] \hline
%&&&&&%&&&
%\\  [10pt] \hline
\end{tabular}

\begin{comment}
\vspace{2mm}
\begin{tabular}{|c|c|c|c|c|c|c|}\hline
&&&&&&\\
Quant. $\diagdown$ Eq.& ({KdV}) & ({EKL}) &   ({EKE}) & (B) &  (NLW) & (NLS) \\  [10pt] \hline
 &&&&&&\\
$\sJ$ & & & & & &  $\left(\begin{array}{cc}0 & 1 \\-1 & 0\end{array}\right)$\\ [10pt] \hline
&&&&&&\\
$\Ham$ & $\frac{1}{2}v_x^2+f(v)$ & & & $\frac{1}{2} \chi_t^2\,+\,W(\chi_x)\,\pm\,\frac{1}{2} \chi_{xx}^2$ & & $ \frac{1}{2} |\psi_x|^2\,+\,F(|\psi|^2)$ \\ [10pt] \hline
&&&&&&\\
$\Impulse$ & $\tfrac{1}{2}\,v^2$ & $\vol \vitsL$ & $-\rho \vits$ & & &  \\ [10pt] \hline
\end{tabular}
\end{comment}

\paragraph{Sturm--Liouville argument}
Assume that  
$$\Ham= \Ham(v,\bu,v_x)\,=\,\En(v,v_x)\,+\,\Ec(v,\bu)\,,
\quad
\frac{\partial^2 \En}{\partial v_x^2}=:\kappa(v)>0\,,\quad 
\nabla^2_{\bu}\Ec=:\bT(v)>0\,,$$
$$\Impulse=\Impulse(\bU)= \tfrac{1}{2}\, \bU\cdot \bB^{-1} \bU\,,\quad \bU^\transp=(v,\bu^\transp)\,,\;\bB^{-1}=\left(\begin{array}{cc} a & \bb^\transp \\ \bb & 0_{N-1}\end{array}\right)\,,$$
with $\kappa(v)>0$ and $\bT(v)$ symmetric definite positive for all $v$, and $\bb\neq 0$.
The profile equations $\Euler(\Ham+\speed\Impulse)[\ubU]+\blambda=0$ equivalently read
$$\left\{\begin{array}{l}\Euler \En[\uv] \,+\,\partial_v \Ec(\uv,\ubu)
 \,+\,\speed\,(a\,\uv\,+\,\bb\cdot \ubu)\,+\,\lambda_1\,=\,0\,,
\\ [8pt]  \nabla_{\bu} \Ec(\uv,\ubu)
\,+\,\speed\,\uv \,\bb\,+\,\check\blambda\,=\,0\,,\end{array}\right.$$
and their integrated version
$$\Legendre(\Ham+\speed\Impulse\,+\,\blambda\cdot \bU)[\ubU]=\mu$$
reads $\Legendre \lag[\uv]\,=\,\mu$, where 
$\lag=\lag(v,v_x;\speed,\blambda)$
is defined by
$$\lag\,=\, \En(v,v_x)\,+\,\Ec(v,\bff(v;\speed,\check\blambda))\,+\,\speed\,(\tfrac{1}{2} a v^2+ v\,\bb\cdot \bff(v;\speed,\check\blambda))\,+\,\lambda_1\,v\,+\,\check\blambda\cdot \bff(v;\speed,\check\blambda)\,,$$
$$\bff(v;\speed,\check\blambda):=-\bT(v)^{-1}\,(\nabla_{\bu} \Ec(v,0)+\speed\, v\, \bb +\check\blambda)\,.$$
Defining
$$\Linvar:=\Hess(\Ham+\speed\Impulse)(\ubU)=\left(\begin{array}{c|c}
\Hess \En[\uv] +\partial_v^2\Ec(\uv,\ubu)+\speed a& (\partial_v\nabla_{\bu} \Ec(\uv,\ubu)+\speed \bb)^\transp
\\ \hline
\partial_v\nabla_{\bu} \Ec(\uv,\ubu) +\speed \bb
& \bT(\uv)\end{array}\right)\,,$$
we see by differentiating with respect to $x$ in the profile equations that $\Linvar\ubU_x=0$, or equivalently
$$\left\{\begin{array}{l}\Hess \En[\uv] \uv_x\,+\,\uv_x\partial_v^2\Ec(\uv,\ubu)\,+\,\ubu_x\cdot \partial_v \nabla_\bu\Ec(\uv,\ubu)\,+\,\speed\,(a \uv_x\,+\,\bb\cdot \ubu_x)\,=\,0\,,\\ [8pt]
\uv_x \partial_v \nabla_\bu\Ec(\uv,\ubu)\,+\,\bT(\uv)\,\ubu_x\,+\,\speed\,\uv_x \,\bb=\,0\,.
\end{array}\right.$$
This can be shown to imply that $\linvar\, \uv_x=0$ with $\linvar:=\Hess \lag[\uv]$. A simpler alternative to show that $\linvar\, \uv_x=0$ consists in differentiating with respect to $x$ in the Euler--Lagrange equation $\Euler \lag[\uv]=0$.
 If in addition  $\En$ depends quadratically on $v_x$, then $\linvar$ is of the form $-\partial_x \kappa(\uv)\partial_x\,+\,q(x)$, where $q(x)$ depends on the
profile $\uv$ -~which depends itself on $(\speed,\blambda,\mu)$~- and on the parameters $\speed,\blambda$. Hence $\linvar$ is  a \emph{Sturm--Liouville} operator with $\Xi$-periodic coefficients. The fact that $\linvar\, \uv_x=0$ and $\uv$ is $\Xi$-periodic (and not constant) implies that $\linvar$ has at least one, and at most two negative eigenvalues
(see for instance \cite[Theorem 5.37]{Teschl}).

\bibliographystyle{plain}

\begin{thebibliography}{}

\end{thebibliography}


\begin{thebibliography}{123}

\bibitem{Benjamin72}
T.~B. Benjamin.
\newblock The stability of solitary waves.
\newblock {\em Proc. Roy. Soc. (London) Ser. A}, 328:153--183, 1972.

\bibitem{Benjamin}
T.~B. Benjamin.
\newblock Impulse, flow force and variational principles.
\newblock {\em IMA J. Appl. Math.}, 32(1-3):3--68, 1984.

\bibitem{BenjaminFeir}
T.~B. Benjamin and J.~E. Feir.
\newblock {D}isintegration of wave trains on deep water .1. {T}heory.
\newblock {\em {J}ournal of {F}luid {M}echanics}, {27}({3}):{417--\&}, {1967}.

\bibitem{SBG-DIE}
{S}. {B}enzoni {G}avage.
\newblock Planar traveling waves in capillary fluids.
\newblock {\em Differential Integral Equations}, 26(3-4):433–478, 2013.

\bibitem{BNR}
S.~Benzoni-Gavage, P.~Noble, and L.~M. Rodrigues.
\newblock {Slow modulations of periodic waves in Hamiltonian PDEs, with
  application to capillary fluids}.
\newblock March 2013.

\bibitem{BR}
S.~Benzoni-Gavage and L.~M. Rodrigues.
\newblock {Co-periodic stability of periodic waves in some Hamiltonian PDEs}.
\newblock In preparation.

\bibitem{BonaSachs}
J.~L. Bona and R.~L. Sachs.
\newblock Global existence of smooth solutions and stability of solitary waves
  for a generalized {B}oussinesq equation.
\newblock {\em Comm. Math. Phys.}, 118(1):15--29, 1988.

\bibitem{BSS}
J.~L. Bona, P.~E. Souganidis, and W.~A. Strauss.
\newblock Stability and instability of solitary waves of {K}orteweg-de {V}ries
  type.
\newblock {\em Proc. Roy. Soc. London Ser. A}, 411(1841):395--412, 1987.

\bibitem{Boussinesq}
J.~Boussinesq.
\newblock Th\'eorie des ondes et des remous qui se propagent le long d'un canal
  rectangulaire horizontal, en communiquant au liquide contenu dans ce canal
  des vitesses sensiblement pareilles de la surface au fond.
\newblock {\em J. Math. Pures Appl.}, 17(2):55--108, 1872.

\bibitem{BronskiJohnson}
J.~C. Bronski and M.~A. Johnson.
\newblock The modulational instability for a generalized {K}orteweg-de {V}ries
  equation.
\newblock {\em Arch. Ration. Mech. Anal.}, 197(2):357--400, 2010.

\bibitem{DeconinckKapitula}
B.~Deconinck and T.~Kapitula.
\newblock On the orbital (in)stability of spatially periodic stationary
  solutions of generalized {K}orteweg-de {V}ries equations.
\newblock 2010.

\bibitem{Gardner93}
R.~A. Gardner.
\newblock On the structure of the spectra of periodic travelling waves.
\newblock {\em J. Math. Pures Appl. (9)}, 72(5):415--439, 1993.

\bibitem{GSS}
M.~Grillakis, J.~Shatah, and W.~Strauss.
\newblock Stability theory of solitary waves in the presence of symmetry. {I}.
\newblock {\em J. Funct. Anal.}, 74(1):160--197, 1987.

\bibitem{Johnson}
Mathew~A. Johnson.
\newblock Nonlinear stability of periodic traveling wave solutions of the
  generalized {K}orteweg-de {V}ries equation.
\newblock {\em SIAM J. Math. Anal.}, 41(5):1921--1947, 2009.

\bibitem{KollarMiller}
R.~Koll{\'a}r and P.~D. Miller.
\newblock Graphical {K}rein signature theory and {E}vans-{K}rein functions.
\newblock 2012.

\bibitem{OhZumbrun10}
M.~Oh and K.~Zumbrun.
\newblock Stability and asymptotic behavior of periodic traveling wave
  solutions of viscous conservation laws in several dimensions.
\newblock {\em Arch. Ration. Mech. Anal.}, 196(1):1--20, 2010.

\bibitem{PegoWeinstein}
R.L. Pego and M.I. Weinstein.
\newblock Eigenvalues, and instabilities of solitary waves.
\newblock {\em Philos. Trans. Roy. Soc. London Ser. A}, 340(1656):47--94, 1992.

\bibitem{PoganScheelZumbrun}
A.~Pogan, A.~Scheel, and K.~Zumbrun.
\newblock Quasi-gradient systems, modulational dichotomies, and stability of
  spatially periodic patterns.
\newblock {\em Diff. Int. Eqns.}, 26:383--432, 2013.

\bibitem{Serre}
D.~Serre.
\newblock Spectral stability of periodic solutions of viscous conservation
  laws: large wavelength analysis.
\newblock {\em Comm. Partial Differential Equations}, 30(1-3):259--282, 2005.

\bibitem{Teschl}
Gerald Teschl.
\newblock {\em Ordinary differential equations and dynamical systems}, volume
  140 of {\em Graduate Studies in Mathematics}.
\newblock American Mathematical Society, Providence, RI, 2012.

\end{thebibliography}

\end{document}